\RequirePackage[l2tabu, orthodox]{nag}
\documentclass[a4paper,11pt,reqno]{amsart}

\usepackage{amssymb,amscd,amsthm}
\usepackage{mathrsfs}

\usepackage[ansinew]{inputenc}
\usepackage[centertags]{amsmath}
\usepackage{epsfig,graphicx}

 \def\RR{{\mathbb R}}  \def\TT{{\mathbb T}}

 \def\ZZ{{\mathbb Z}}

\def\cF{\mathcal{F}}

\newtheorem*{teo*}{Theorem}
\newtheorem*{prop*}{Proposition}
\newtheorem*{cor*}{Corollary}
\newtheorem*{goal*}{Goal}

\newtheorem*{teoA'}{Theorem A'}

\newtheorem{teo}{Theorem}[section]
\newtheorem{thm}[teo]{Theorem}

\newtheorem{quest}[teo]{Question}
\newtheorem{cor}[teo]{Corollary}

\newtheorem{prop}[teo]{Proposition}

\newcommand{\bi}{\begin{itemize}}
\newcommand{\ei}{\end{itemize}}

\theoremstyle{definition}

\theoremstyle{remark}

\newtheorem{remark}[teo]{Remark}

\numberwithin{equation}{section}

\newcounter{notes}%

\author[R. Potrie]{Rafael Potrie}
\address{CMAT, Facultad de Ciencias, Universidad de la Rep\'ublica, Uruguay}
\curraddr{IAS, Princeton, NJ, USA}
\urladdr{www.cmat.edu.uy/$\sim$rpotrie}
\email{rpotrie@cmat.edu.uy}

\title[Partially hyperbolic dynamics]
{Partially hyperbolic dynamics and 3-manifold topology}
\thanks{The author was partially supported by CSIC-618, FCE-135352 the Minerva Research Foundation Membership Fund and the grant NSF DMS-1638352. This work was written while the author was a Von Neumann fellow at the Institute for Advanced Study and wants to acknowledge the excellent working conditions and environment. Comments of Silvia Ghinassi, Mariana Haim, Santiago Martinchich, Mart\'in Reiris and Jan Vonk were very helpful in the writing of the note. The autor wishes to thanks particularly the referees who provided a lot of helpful input to improve the paper.}

\begin{document}

\maketitle
%
%\begin{abstract}
%x
%\end{abstract}
%  
\medskip

%%%%%%%%%%%%%%%%%%%%%%%%%%%%%%%%%%%%%%%%%%%%%%%%%%%
\section{Introduction}

The hairy ball theorem implies that the two dimensional sphere cannot admit a vector field without singularities. This is just an example of a restriction imposed by the topology of a phase space to the possible dynamics it can support.

In this note we would like to present and relate two results in this direction. These are about restrictions imposed by the topology of certain 3-manifolds on the dynamics it can support. 

The first result we will present corresponds to the theory of Anosov flows and was proved by Margulis \cite{Margulis} when he was still an undergraduate student in an appendix to a paper of Anosov and Sinai \cite{AnosovSinai}. The result was later revisited by Plante and Thurston \cite{PlanteThurston} who extended its scope and proposed a different approach that used some finer properties of foliations. %This is Theorem \ref{t.main} below. 

A flow $\phi_t : M \to M$ generated by a (smooth) vector field $X$ on a closed manifold $M$ is said to be an \emph{Anosov flow} if there is a continuous $D\phi_t$-invariant splitting of the tangent bundle $TM = E^s \oplus \RR X \oplus E^u$ satisfying that there is $t_0>0$ so that for every $v^\sigma \in E^{\sigma}$ ($\sigma = s,u$) a unit vector we have that $\|D\phi_{t_0} v^s\|< 1 < \|D\phi_{t_0} v^u\|$. This immediately implies that stable vectors (i.e. those in $E^s$) are contracted exponentially by $D\phi_t$ while unstable vectors (i.e. those in $E^u$) are expanded exponentially fast by $D\phi_t$.

\begin{figure}[ht]
\begin{center}
\includegraphics[scale=0.35]{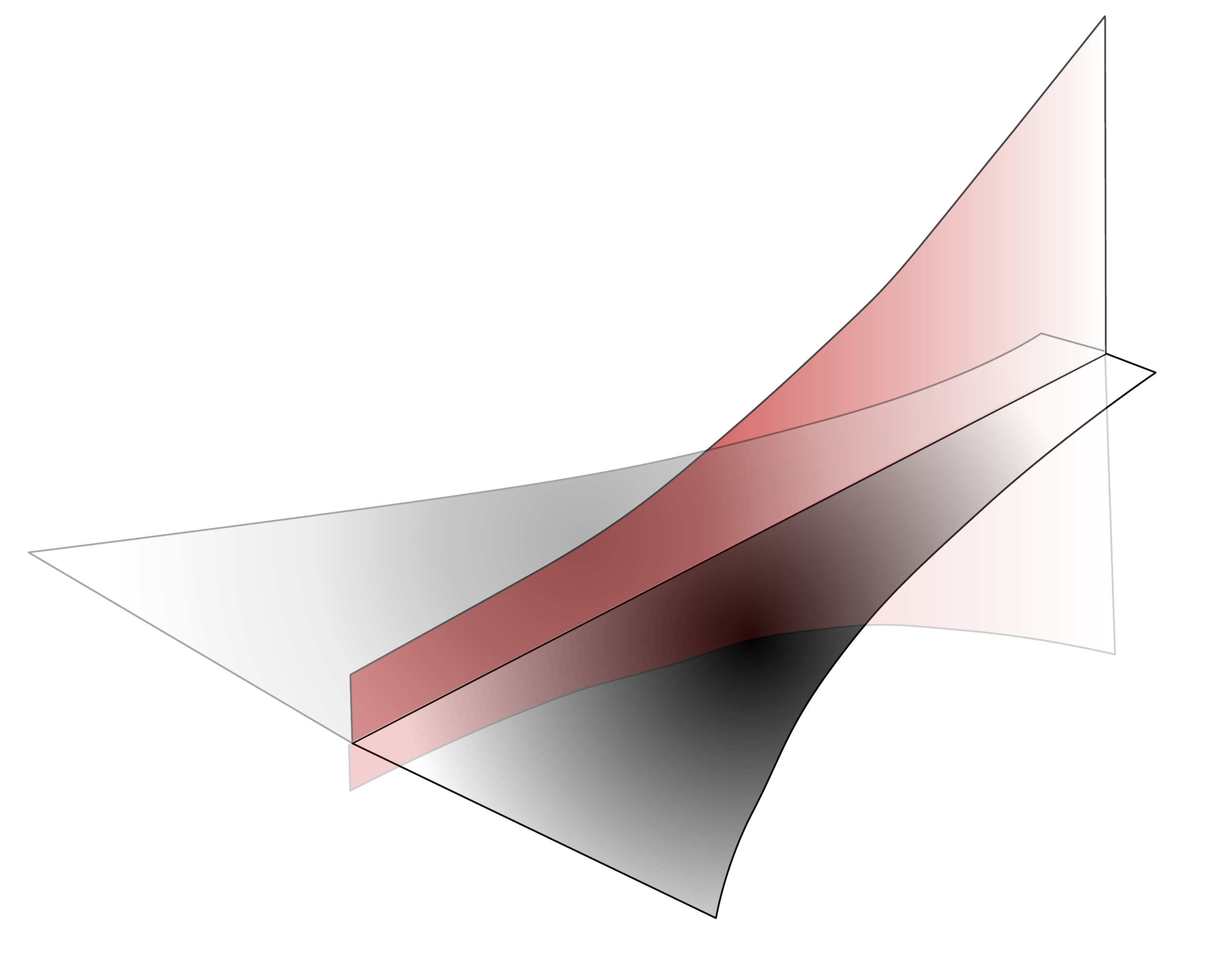}
\begin{picture}(0,0)
%\put(-310,20){$p$}
%\put(-211,117){$t_{F_0}$}
%\put(-165,15){$t_{L_0}$}
%\put(-22,30){$L_0$}
%\put(-35,138){$F_0$}
%\put(-36,88){$\gamma_i L$}
%\put(-250,52){$u(p)$}
%\put(-280,50){\small$\cW^{u}(y)$}
%\put(-60,99){$x$}
%\put(-238,160){$\cW^{s}(x)$}
%\put(-290,120){$\cW^{s}(x')$}
\end{picture}
\end{center}
\vspace{-0.5cm}
\caption{{\small The local aspect of orbits of Anosov flow.}}\label{fig.aflow}
\end{figure}

Examples of Anosov flows include geodesic flows in negative curvature \cite{Anosov} as well as suspensions of certain toral automorphisms. Their definition goes back at least to the paper of Anosov and Sinai \cite{AnosovSinai} where they extracted the properties from geodesic flows in negative curvature needed to obtain ergodicity. We point out that in 3-manifolds we know that Anosov flows contain the space of robustly transitive flows (i.e. flows so that every perturbation has some dense orbit), see \cite{Doering,BDV}. The result by Margulis and Plante-Thurston says that if a 3-manifold admits an Anosov flow then its fundamental group has exponential growth (see Theorem \ref{t.main}) and implies in particular that 3-manifolds such as the sphere $S^3$ or the three torus $\mathbb{T}^3= \mathbb{R}^3/_{\mathbb{Z}^3}$ do not admit such flows.

The second result is more recent and essentially due to Burago and Ivanov \cite{BI}. The result gives some obstructions for some mapping classes of certain 3-manifolds to admit partially hyperbolic diffeomorphisms. A diffeomorphism $f: M \to M$ is said to be \emph{partially hyperbolic} if the tangent space $TM$ splits as a direct sum of non-trivial continuous subbundles $E^s \oplus E^c \oplus E^u = TM$ which are $Df$-invariant and satisfy that there is some $\ell>1$ so that for every $x\in M$, if $v^\sigma \in E^\sigma(x)$ ($\sigma = s,c,u$) are unit vectors, then: 

\begin{equation}\label{eq:PH}
 \|Df^\ell v^s \| < \min \{ 1, \|Df^\ell v^c \|\} \quad \text{  and  }  \quad
  \|Df^\ell v^u \| > \max \{ 1, \|Df^\ell v^c \| \}. 
\end{equation}

Naturally, time one maps of Anosov flows (i.e. the diffeomorphism $f(x) = \phi_1(x)$ where $\phi_t$ is an Anosov flow) are examples where $E^c=\RR X$. Many homogeneous dynamics (namely, those which have positive entropy) are partially hyperbolic. Some of them are not time one maps of Anosov flows; for instance, the action of a matrix $A \in \mathrm{SL}(3,\ZZ)$ in the 3-torus $\mathbb{T}^3$ with three different real eigenvalues is such an example. Being partially hyperbolic is an open property in the $C^1$-topology, so it is possible to make $C^1$-small perturbations to the mentioned examples to obtain new examples. 

In dimension 3, the result of Burago-Ivanov we will explain implies that if a diffeomorphism $f: M \to M$ is partially hyperbolic and the manifold does not have fundamental group of exponential growth, then $f$ cannot be homotopic to the identity. See Theorem \ref{t.obstructionph} for a precise statement. 

The connection between these two results will allow us to briefly comment on the classification of partially hyperbolic diffeomorphisms in 3-manifolds, referring the interested reader to recent surveys such as  \cite{CHHU, HP, PotrieICM, BFFP, BFP} for a more complete presentation. 

%%%%%%%%%%%%%%%%%%%%%%%%%%%%%%%%%%%
\section{Anosov flows and foliations}\label{s.AFlows}

Consider an Anosov flow $\phi_t: M \to M$ in a closed manifold $M$. The definition requires that the differential of the flow preserves some geometric structure. It could seem hard to check that a flow is Anosov, but it is important to remark that the existence of the $D\phi_t$ invariant bundles follows from the existence of a way more flexible structure, namely that of \emph{invariant cone-fields} which we will not define but just point out that these are objects that are robust (i.e. if a system has invariant cone-fields then this is true in a $C^1$-open neighborhood) and somewhat easy to check (i.e. a computer can check whether a system is Anosov). 

The importance of this infinitesimal condition is that it can be pushed into the manifold in a way that one obtains objects whose dynamics mimic the dynamics of the differential map. 

\begin{thm}[Stable manifold theorem]\label{thm-stablemfd}
The bundles $E^s$ and $E^s \oplus \RR X$ are uniquely integrable. 
\end{thm}

We need to say some words to explain what we mean. Let $E \subset TM$ be a $k$-dimensional subbundle of the tangent bundle of $M$. We say that $E$ is uniquely integrable if through every point $x \in M$ there is a $k$-dimensional submanifold $S_x$ everywhere tangent to $E$ such that every curve tangent to $E$ through $x$ is completely contained in $S_x$.

The same is true for $E^u$ and $\RR X \oplus E^u$ by applying the theorem to $\phi_{-t}$. Notice that even when $\mathrm{dim}(E^s)=1$ showing its unique integrability is not obvious since these bundles are typically no better than H\"{o}lder continuous, so one needs to appeal to dynamics to get this kind of results. 

The proof of this result when $\mathrm{dim}(E^{s})=1$ is not complicated and we will now sketch it: Assume by contradiction that there are two different curves $\gamma_1$ and $\gamma_2$ everywhere tangent to $E^s$ which separate at a point $x \in M$.  Note that by considering $\phi_t$ with large $t$ we get that the curves $\gamma_1$ and $\gamma_2$ decrease their length exponentially fast, however, their transverse distance, measured along the direction $\RR X \oplus E^u$ cannot decrease (see figure \ref{fig.aflow}), which provides a contradiction. To show unique integrability of $E^s \oplus \RR X$ one just needs to flow the integral curves of $E^s$ by the flow (whose defining vector field is smooth, so uniquely integrable).

\begin{figure}[ht]
\begin{center}
\includegraphics[scale=0.55]{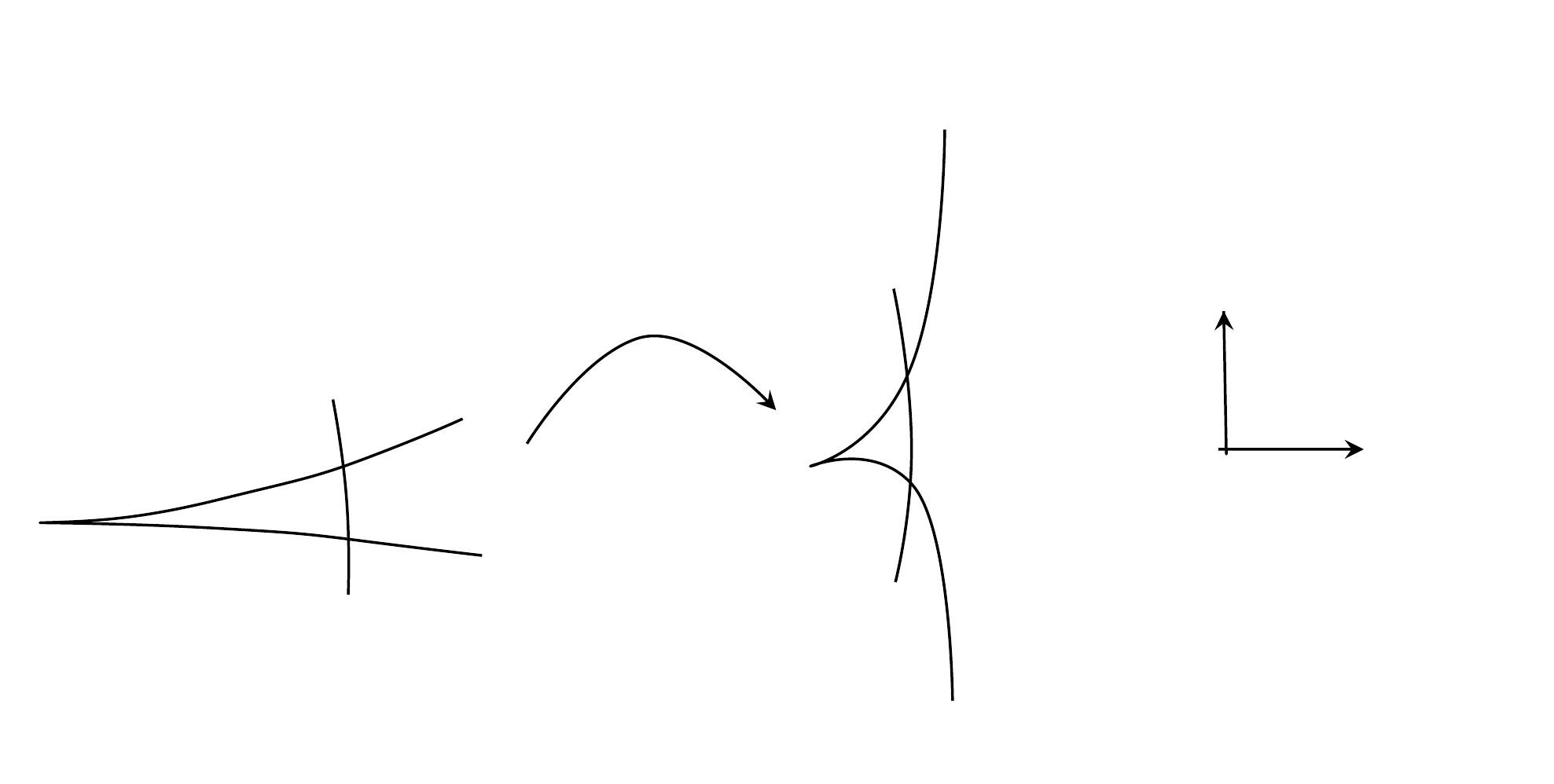}
\begin{picture}(0,0)
\put(-40,64){$E^{s}$}
\put(-85,99){$\RR X \oplus E^{u}$}
\put(-191,100){$\phi_t$}
%\put(-165,15){$t_{L_0}$}
%\put(-22,30){$L_0$}
%\put(-35,138){$F_0$}
%\put(-36,88){$\gamma_i L$}
%\put(-250,52){$u(p)$}
%\put(-280,50){\small$\cW^{u}(y)$}
%\put(-60,99){$x$}
%\put(-238,160){$\cW^{s}(x)$}
%\put(-290,120){$\cW^{s}(x')$}
\end{picture}
\end{center}
\vspace{-0.5cm}
\caption{{\small Flowing forward different curves tangent to $E^s$ through a given point gives a contradiction.}}\label{fig.aflow}
\end{figure}
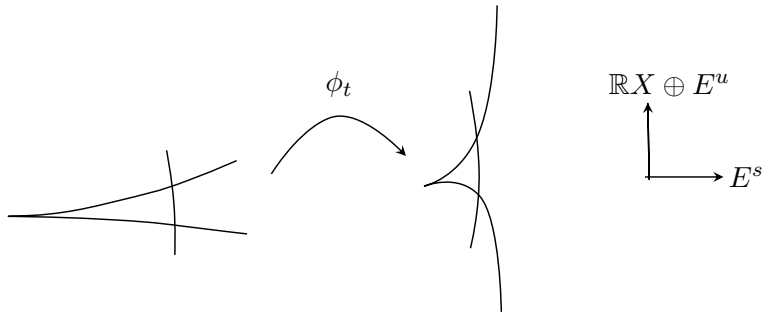

The curves tangent to $E^s$ and the surfaces tangent to $E^s \oplus \RR X$ form what we call the \emph{strong stable} and \emph{weak stable} foliations, which together with their dual strong unstable and weak unstable foliations are one of the main tools to understand the dynamics and geometry of Anosov systems. Let us just state an easy fact about these that we will use later: 

\begin{prop}\label{prop-noncompact}
There are no closed submanifolds $N$ of $M$ tangent to $E^s \oplus \RR X$. 
\end{prop}

This is a direct consequence of the fact that the time $t_0$ map $\phi_{t_0}$ of the flow would be a diffeomorphism of the compact manifold $N$ whose derivative is everywhere contracting volume, which is just impossible.

Recall that a \emph{foliation by surfaces} of a 3-manifold $M$ is a partition of $M$ by injectively immersed $C^1$-surfaces (called \emph{leaves}) that locally look like horizontal planes $\RR^2 \times \{t\}$ sitting inside $\RR^3$  (i.e. there are charts sending leaves to horizontal planes, see figure \ref{fig.fol}). An example of a foliation by surfaces would be the weak stable foliation of an Anosov flow in a 3-manifold.

\begin{figure}[ht]
\begin{center}
\includegraphics[scale=0.35]{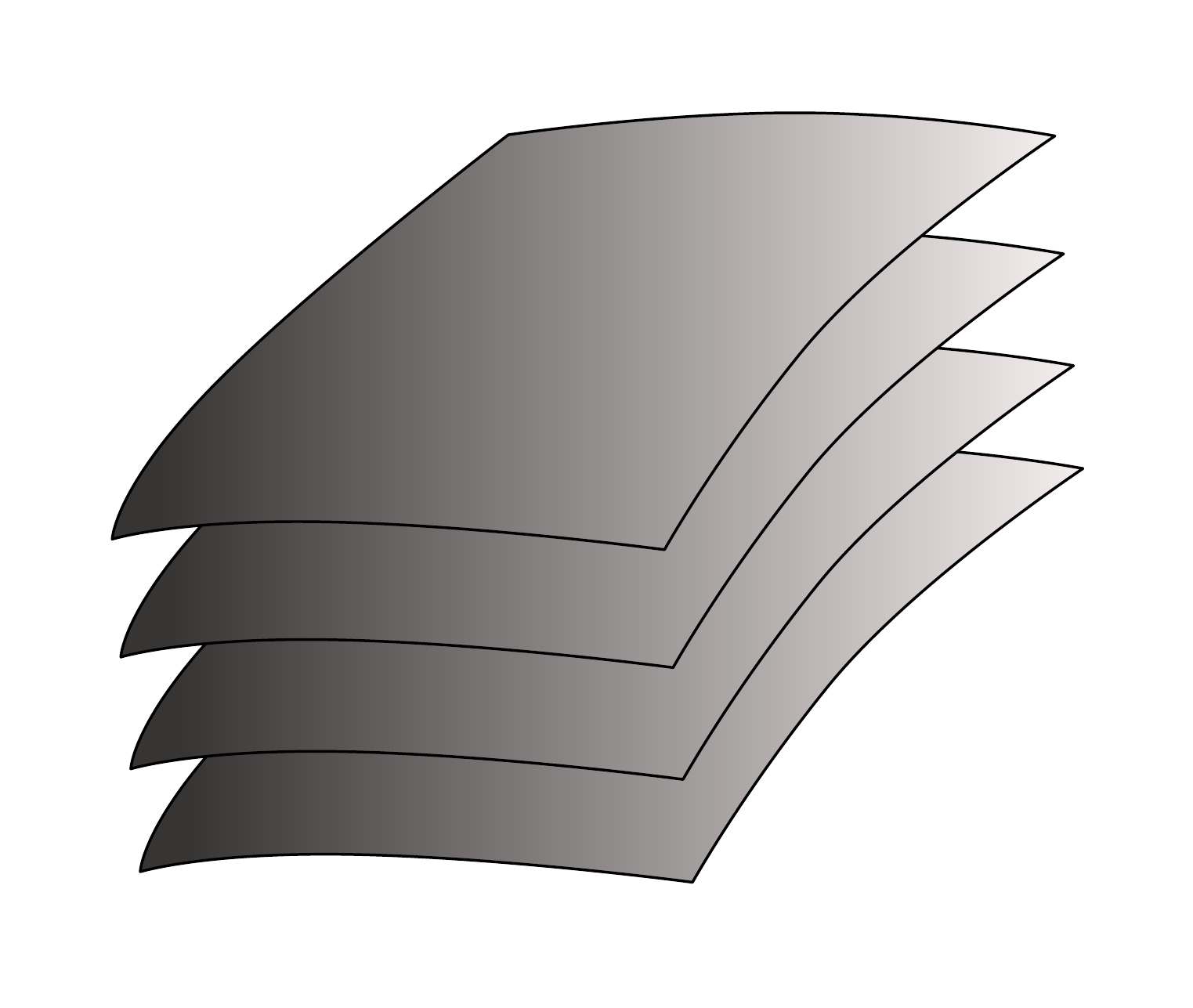}
\begin{picture}(0,0)
\end{picture}
\end{center}
\vspace{-0.5cm}
\caption{{\small A foliation locally looks as a stack of surfaces.}}\label{fig.fol}
\end{figure}

A \emph{transversal} to a foliation is an embedded circle which is everywhere transverse to the leaves of $\cF$. Notice that if $M$ is compact there are always transversals since a transverse curve intersecting a foliation box twice can be closed into a transversal. See figure \ref{fig.tran}. 

\begin{thm}[Novikov]\label{thm-Novikov} 
Let $\cF$ be a foliation by surfaces in a closed 3-manifold. Assume that there is a transversal $\gamma$ to $\cF$ which is homotopically trivial. Then $\cF$ has a closed leaf. 
\end{thm}

We will not prove this beautiful result which has several expositions. In fact, Novikov's result is much stronger and implies the existence of what are known as \emph{Reeb components}. One should think that in 3-manifolds compact leaves (or Reeb components) of foliations play the role that singularities play in vector fields in surfaces, and therefore Novikov's theorem acts as the Poincare-Bendixon's theorem in this setting\footnote{Even if much deeper, there is a part of the proof of Novikov's theorem (which is indeed enough to rule out homotopically trivial transversal loops for Anosov flows) that is very much modelled in the proof of Poincare-Bendixon's theorem. It is known as \emph{Haefliger's} argument: using the transverse loop, one constructs a disk whose boundary is transverse to the foliation and which is in general position; studying the induced flow on the disk is enough to find a configuration which is not compatible with Anosov flows, and such that with much more work produces a Reeb component. We note that for the partially hyperbolic case to be treated later, the full version of Novikov's theorem is important.\label{footnote}}. We refer the reader to  \cite[Chapter 4]{Calegari} for a friendly account on foliations in 3-manifolds. 

%%%%%%%%%%%%%%%%%%%%%%%%%%%%%%%%%%
\section{Margulis/Plante-Thurston's result} 
In the late 60's Margulis showed the following beautiful result: 

\begin{thm}\label{t.main}
Let $M$ be a closed 3-dimensional manifold admitting an Anosov flow $\phi_t$, then, the fundamental group of $M$ grows exponentially. 
\end{thm}

A finitely generated group $\Gamma$ has \emph{exponential growth} if for some finite generating set $F \subset \Gamma$ it follows that the number of different group elements that can be written as a product of at most $n$ elements of $F \cup F^{-1}$ grows exponentially with $n$.  This is independent of the finite generating set $F$. 

In a closed manifold $M$ the fundamental group has exponential growth if and only if the volume of a ball or radius $R$ in $\widetilde M$, the universal cover of $M$, grows exponentially with respect to $R$. This means, if $\pi: \widetilde M \to M$ is the universal covering map, and we consider in $\widetilde M$ the metric induced by $\pi$, then there is a point $x \in \widetilde M$ and constants $c,\delta>0$ so that: 

\begin{equation}\label{eq:expgrowth}
\mathrm{vol}(B(x, R)) > c e^{\delta R},
\end{equation}

\noindent where $B(x,R)$ denotes the ball of center $x$ and radius $R$ in $\widetilde M$. To see the equivalence one just needs to find a compact fundamental domain in the universal cover and note that its volume must be finite, this way one can cover the ball of radius $R$ by deck transformations of bounded size and compare the growth of the volume of the ball with the growth of the fundamental group as a finitely generated group. This is the definition we will use to prove Theorem \ref{t.main}. 

It is an easy exercise to show that, up to changing the constants $c,\delta$, the definition is independent on the point $x \in \widetilde M$ as well as on the metric one pulls back from $M$, so that this is indeed a topological property of $M$ which in fact only depends on its fundamental group. Under this assumption we say that $M$ has \emph{exponential growth of fundamental group}. 

The proof by Margulis \cite{Margulis} is direct and independent of any deep result in foliation theory (even if the foliations are used crucially). Later, Plante and Thurston \cite{PlanteThurston} gave a more conceptual proof that works for general codimension one Anosov flows\footnote{i.e. those whose stable or unstable bundle is one-dimensional} and uses some deeper results in foliation theory. The proof we shall present here has ingredients from both organised in a way that will lead us naturally to the generalisation of these arguments to the classification problem of partially hyperbolic diffeomorphisms in dimension 3. 

We emphasize the following fact: In dimension 2, the hairy ball theorem, or the Poincare-Hopf index theorem imply that admitting a continuous subbundle is already enough to get some topological obstruction (i.e. only the two torus and the Klein bottle admit a continuous splitting of the tangent bundle). However, this is not the case in dimension 3; up to double cover, every closed 3-manifold has trivial tangent bundle. That is, $TM \cong M \times \RR^3$, therefore the existence of a splitting of the tangent bundle cannot be an obstruction by itself. It will be finer properties of the foliations that these bundles integrate, namely, the non-existence of compact leaves, that will come handy for this issue. 

\begin{remark}
At the time that Margulis proved this theorem, the only known examples of Anosov flows in closed 3-manifolds were the geodesic flows in negative curvature (and its finite lifts), and the suspension flows of linear hyperbolic automorphisms of tori. Later, new examples started to appear, especially in dimension 3 (see \cite{Barthelme}). The study of the geometry and topology of Anosov flows in dimension 3 has grown tremendously since these pioneering work. We refer the reader to \cite{Barthelme} for a survey of the main results with several of the key ideas.   
\end{remark}

\begin{remark}
Theorem \ref{t.main} implies the following result which also admits a more elementary proof just using Lefschetz index. If $f: \TT^2 \to \TT^2$ is an Anosov diffeomorphism\footnote{An Anosov diffeomorphism $g: M \to M$ is such that $Dg$ preserves a splitting $TM=E^{s} \oplus E^u$ so that vectors in $E^s$ are uniformly contracted and vectors in $E^u$ are uniformly expanded as in \S~\ref{s.AFlows} .} of a two-torus, then, the action of $f$ in homology is hyperbolic, meaning that it has no eigenvalue of modulus $1$. An interesting challenge could be to prove this statement after (or before!) reading the proof below. 
\end{remark}

%%%%%%%%%%%%%%%%%%%%%%%%%%%%%%%%%%
\section{The proof}\label{s.theproof} 
We provide here a quick proof of Theorem \ref{t.main} based on the original arguments, but probably with a more modern viewpoint.  The goal is motivating tools that allow understanding the interaction between topology and dynamics. 

An easy consequence of Theorem \ref{t.main} is the non-existence of Anosov flows in the sphere $S^3$. This can also be shown quite directly by a shortcut in the same argument: Assume that $\phi_t: S^3 \to S^3$ is an Anosov flow. Consider $\cF^{ws}$ the weak stable foliation of $\phi_t$ given by Theorem \ref{thm-stablemfd}. By Novikov's compact leaf theorem (see Theorem \ref{thm-Novikov}) we know that every foliation by surfaces in $S^3$ must have a compact leaf, this contradicts Proposition \ref{prop-noncompact}. 

With these elements in hand, we are ready to give the proof. The reader not comfortable with the basics of algebraic topology can use as a model the 3-torus $\mathbb{T}^3 = S^1 \times S^1 \times S^1 = \RR^3/_{\ZZ^3}$ where integer translations are deck transformations of $\RR^3$ its universal cover.  The theorem implies that $\mathbb{T}^3$ does not admit Anosov flows, and the difficulty of proving this case is the same as the general case. Here, one will have that $\widetilde{\mathbb{T}^3}= \RR^3$ with the Euclidean metric (so balls do not have exponential growth of volume by a direct computation).

\begin{proof}[Proof of Theorem \ref{t.main}] Let $\pi: \tilde M \to M$ be the universal cover and lift $\phi_t$ to a flow $\tilde \phi_t: \widetilde M \to \widetilde M$. Let $\widetilde{\cF^{ws}}$ the lift of the weak stable foliation to $\widetilde M$. 

Consider an arc $J$ tangent to the bundle $\tilde E^u$ (the lift of $E^u$). The arc $J$ is transverse to $\widetilde{\cF^{ws}}$. Since the foliation is invariant under $\tilde \phi_t$ and the arc $J$ maps to another arc tangent to $\tilde E^u$ we deduce that the arc $\tilde \phi_t(J)$ cannot intersect the same foliation box twice, since that would allow to construct a transversal to $\cF^{ws}$ which is homotopically trivial, contradicting Theorem \ref{thm-Novikov} and Proposition \ref{prop-noncompact}. See figure \ref{fig.tran}.  

\begin{figure}[ht]
\begin{center}
\includegraphics[scale=0.35]{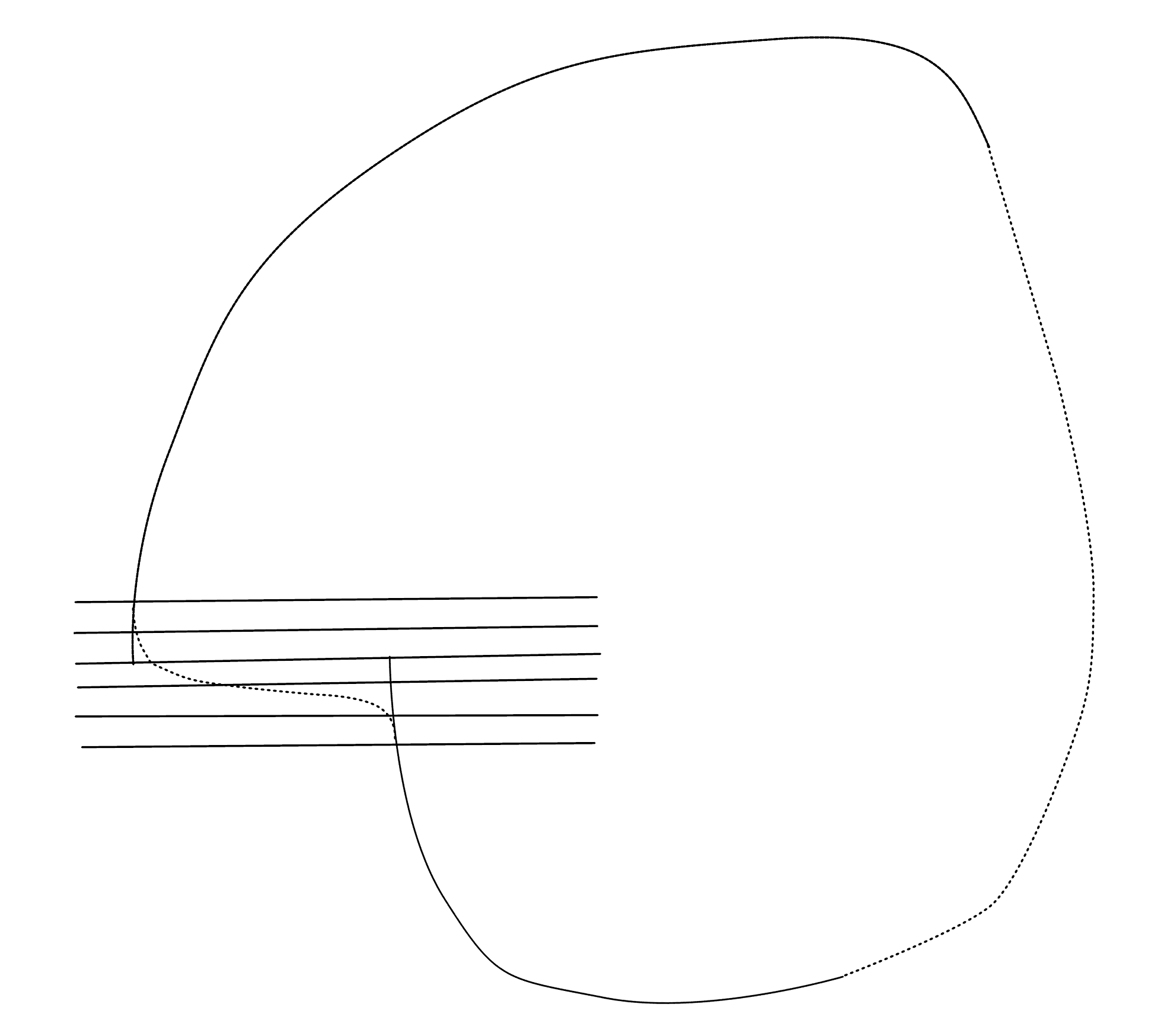}
\begin{picture}(0,0)
%\put(-310,20){$p$}
%\put(-211,117){$t_{F_0}$}
%\put(-165,15){$t_{L_0}$}
%\put(-22,30){$L_0$}
%\put(-35,138){$F_0$}
%\put(-36,88){$\gamma_i L$}
%\put(-250,52){$u(p)$}
%\put(-280,50){\small$\cW^{u}(y)$}
%\put(-60,99){$x$}
%\put(-238,160){$\cW^{s}(x)$}
%\put(-290,120){$\cW^{s}(x')$}
\end{picture}
\end{center}
\vspace{-0.5cm}
\caption{{\small If a positively transverse curve intersects a leaf twice one can construct a closed transversal.}}\label{fig.tran}
\end{figure}

Foliation boxes have uniform size since they can be pulled back from $M$ which is compact. One deduces that there exists a uniform constant $c_0>0$ so that:

$$ \mathrm{vol}(B(\tilde \phi_t(J),1)) > c_0\mathrm{length}(\tilde \phi_t(J)). $$

\noindent where $B(X,r)$ denotes the set of points in $\tilde M$ at distance less than $r$ from $X$. Moreover, since $J$ is tangent to $E^u$ there are positive constants $c_1, \delta>0$ so that $\mathrm{length}(\tilde \phi_t(J)) > c_1 e^{\delta t}$. Putting this together, one gets:

$$ \mathrm{vol}(B(\tilde \phi_t(J), 1)) > c_0 c_1 e^{\delta t}. $$

We will now show that there is a constant $c_2>0$ so that if $x_0 \in J$, then $\tilde \phi_t(J)$ is contained in $B(x_0, R_t)$ where $R_t \leq c_2 t + \mathrm{diam}(J)$. This is obtained by computing, for $x \in J$ 

\begin{equation}\label{eq:diameter} d(x_0, \tilde \phi^t (x)) \leq d(x_0, x) + d(x, \tilde \phi_t(x)) \leq \mathrm{diam}(J) + c_2 t, 
\end{equation}

\noindent where $c_2$ is a bound of the norm for the vector field generating $\phi_t$. 

This implies that $ B(\tilde \phi_t(J), 1) \subset B( x_0, R_t +1)$ and therefore, taking $\hat \delta = \frac{\delta}{c_2}$ and $c_3= e^{-\hat \delta({diam}(J))}$ and  we get 

$$ \mathrm{vol}(B(x_0, R_t + 1)) > c_0 c_1 c_3 e^{\hat \delta (R_t+ 1)} $$ 

\noindent which gives \eqref{eq:expgrowth} and completes the proof. 
\end{proof}

Margulis proof is more elementary since it does not use any deep result about foliations, however, it depends crucially on the fact that the weak stable/unstable foliation is \emph{complete} in the sense that a weak stable/unstable leaf is the union of the strong stable/unstable manifolds through points of a given orbit. This fact fails when one goes to the partially hyperbolic setting. This property is used by Margulis to construct by hand the universal cover of $M$ and compute its volume growth. 

 The proof of Plante and Thurston is much more similar to the one we present here, only that instead of computing volume they construct many loops that they show to be pairwise nonhomotopic. For this, they use Haefliger's argument (cf. footnote~\ref{footnote}). In particular, as the proof presented here and in contrast with Margulis proof it only needs one of the two foliations and that is why it extends to codimension one Anosov flows. But the importance here is that this line of reasoning does not depend on understanding the internal structure of the codimension one foliation, and so is well suited to be extended in other contexts. 

%%%%%%%%%%%%%%%%%%%%%%%%%%
\section{Classification of partially hyperbolic systems}
We will now come to the problem of understanding the structure of general partially hyperbolic systems in 3-dimensional manifolds by modelling the questions and ideas in the work done in the previous section. 
%
%Let us first define partially hyperbolic diffeomorphisms. We say that a diffeomorphism $f : M \to M$ is partially hyperbolic if the tangent space $TM$ splits as a direct sum of non-trivial continuous subbundles $E^s \oplus E^c \oplus E^u = TM$ which are $Df$-invariant and satisfy that there is some $\ell>1$ so that for every $x\in M$, if $v^\sigma \in E^\sigma(x)$ ($\sigma = s,c,u$) are unit vectors, then: 
%
%\begin{equation}\label{eq:PH}
% \|Df^\ell v^s \| < \min \{ 1, \|Df^\ell v^c \|\} \quad \text{  and  }  \quad
%  \|Df^\ell v^u \| > \max \{ 1, \|Df^\ell v^c \| \}. 
%\end{equation}
%
%Examples include many homogeneous dynamics (all of those with positive entropy) such as linear automorphisms of tori given by matrices $A \in \mathrm{SL}(d,\mathbb{Z})$ with some eigenvalue of modulus different from one, or certain translations in compact quotients of simple Lie groups (for instance, the time one map of the geodesic flow in constant negative curvature surfaces). But also other geometric constructions belong to this class (e.g. time one maps of frame flows in negatively curved manifolds), as well as time one maps of Anosov flows defined before. Partially hyperbolic diffeomorphisms form a $C^1$-open class of diffeomorphisms and some robust dynamical behaviour typically forces some sort of partial hyperbolicity.  We refer the reader to the surveys provided in \S \ref{s.furtherreading} for discussions on these aspects. 
%
Here we shall concentrate on the following questions of current research interest which can be considered as continuations of the problem discussed above for Anosov flows: 

\begin{quest}
Which $3$-manifolds admit partially hyperbolic diffeomorphisms? Which isotopy classes? Are these similar in some way to the known examples? 
\end{quest}

We refer the reader to \cite{CP} for a general expositions on the basic facts about partially hyperbolic systems as well as a long list of examples. Here we will concentrate in a few relevant aspects specific to 3-dimensions. 

A main difference which makes studying partially hyperbolic diffeomorphisms much harder than Anosov flows is that even if the strong bundles $E^s$ and $E^u$ still integrate uniquely into $f$-invariant foliations (essentially by the same argument as in the Anosov flow case), this is no longer true for the  \emph{center stable} and \emph{center unstable} bundles $E^{cs} = E^s \oplus E^c$ nor $E^{cu}=E^c \oplus E^u$.  This makes the study of partially hyperbolic diffeomorphisms much harder. At the beginning of its exploration, the topological study of these systems assumed the existence of such foliations under the concept of \emph{dynamical coherence} since all the known examples had them. We say that a partially hyperbolic diffeomorphism is \emph{dynamically coherent} if there are $f$-invariant foliations tangent respectively to $E^{cs}$ and $E^{cu}$. 

A recent breakthrough result by Burago and Ivanov \cite{BI} provided a tool for avoiding such an undesirable hypothesis\footnote{The reason it is undesirable is that it is not easy to check, and that several examples have appeared where it is known not to hold.}. 

\begin{thm}\label{t.BI}
Up to finite cover, there is a Reebless foliation $\cF$ transverse to the unstable direction $E^{u}$.  
\end{thm}

This implies by iterating backwards that one can choose the foliation to be as close to tangent to $E^{cs}$ as desired, but does not imply dynamical coherence, as in the limit the leaves could merge together forming what is called a \emph{branching foliation} which is an incredibly useful tool for the study of partially hyperbolic diffeomorphisms but that we will avoid to discuss here. We note here that the proof of Theorem \ref{t.BI} depends very strongly on the fact that $E^c$ and $E^u$ are one dimensional; indeed, one can not expect a similar result if $E^c$ has higher dimensions. % since even if it where smooth, it could have a Frobenius bracket with non-trivial coordinate in the $E^u$ direction. 

In fact, to show the result it is enough to show that there exists a foliation transverse to $E^u$ since the non-existence of Reeb components follows from the fact that there are no closed curves tangent to $E^u$. This just follows from the fact that a flow transverse to a Reeb component must have a closed orbit. This beautiful observation from \cite{BI} allows them to treat the problem locally and obtain this global information. 

Theorem \ref{t.BI} has the following consequence which is the first known topological obstruction for the existence of partially hyperbolic diffeomorphisms: 

\begin{cor}[Burago-Ivanov] The sphere $\mathbb{S}^3$ does not admit partially hyperbolic diffeomorphisms. 
\end{cor}

The proof of Theorem \ref{t.main} in \S \ref{s.theproof} has as a moral that to expand a one-dimensional foliation transverse to a two-dimensional foliation in a 3-manifold one needs space. This moral quite extends to the diffeomorphism case, only that diffeomorphisms can wrap the manifold onto themselves and then obtain expansion without much space. 

For instance, a matrix in $\mathrm{SL}(3,\mathbb{Z})$ with real eigenvalues and at least one larger than $1$ induces a partially hyperbolic diffeomorphism on $\mathbb{T}^3$.  The volume growth of the universal cover $\mathbb{R}^3$ of $\mathbb{T}^3$ is just polynomial. The reason is that the action of $f$ itself already gives the foliation space to expand. In the proof of Theorem \ref{t.main} this appears in the crucial use of the fact that $\phi_t$ is a flow (or equivalently, that its time one map is homotopic to the identity) which gives equation \eqref{eq:diameter}. 

With essentially the same proof as for Theorem \ref{t.main} by replacing the stable manifold theorem with Theorem \ref{t.BI} one can obtain the following result which provides obstructions for the mapping classes which admit partially hyperbolic diffeomorphisms:

\begin{thm}\label{t.obstructionph}
If $f: M \to M$ is a partially hyperbolic diffeomorphism of a closed 3-dimensional manifold and $\hat M_f$ is the mapping torus of $f$, then the fundamental group $\pi_1(\hat M_f)$ of $\hat M_f$ has exponential growth. 
\end{thm}

Recall that the \emph{mapping torus} of a map $F : X \to X$ is the space $X \times [0,1]/_{\sim}$ where one identifies $(x,1) \sim (F(x),0)$ for all $x$. It depends only on the homotopy class of the map $F$, and produces a smooth manifold if $X$ is a manifold and $F$ a diffeomorphism (so that the equivalence with volume growth still holds). 

But Theorem \ref{t.BI} is indeed stronger, since it can also provide further obstructions thanks to the well developed theory of Reebless foliations. There are manifolds with exponential growth of fundamental group known not to admit foliations without compact leaves, including some hyperbolic 3-manifolds (see e.g. \cite[Example 4.4.6]{Calegari}).  These provide also obstructions to the existence of partially hyperbolic diffeomorphisms. Up to recently, these were more or less all the known obstructions to the existence of partially hyperbolic diffeomorphisms. At the moment of this writing, we do not know any manifold with exponential growth of fundamental group which admits a partially hyperbolic diffeomorphism but does not admit an Anosov flow. But lots of developments have been made recently that give hope that the understanding of partially hyperbolic diffeomorphisms is not far from the understanding of Anosov flows. 

\section{Further discussion} 
As mentioned, the obstruction given by Theorem \ref{t.obstructionph} is not sharp, so it makes sense to see to what extent one can characterise the homotopy classes of diffeomorphisms of $3$-manifolds admitting partially hyperbolic diffeomorphisms. It turns out that only very recently examples in new isotopy classes where found \cite{BGHP}. In these examples, new features of partially hyperbolic systems were exposed, in particular, the global nature of dynamical coherence is now better understood. 

But somehow, all examples we know build in some way or the other on some Anosov system. The examples in \cite{BGHP} are constructed by using the cone-field criterium to guarantee partial hyperbolicity together with a careful understanding of the global structure of the invariant bundles. This way, it is possible to construct diffeomorphisms of the manifold which respect transversalities between the bundles, and this allows to create new partially hyperbolic diffeomorphisms in new isotopy classes. These kinds of constructions are still in their infancy, and it is likely that new examples can be created using these ideas. Nonetheless, there are some manifolds and isotopy classes of diffeomorphisms where the partially hyperbolic dynamics seem amenable to classification, notably hyperbolic and Seifert 3-manifolds \cite{BFFP, BFP}. A notion of \emph{collapsed Anosov flow} has been proposed recently that may account for all new examples, and which needs to be tested against new potential constructions \cite{BFP}. 

In higher dimensions, Anosov systems are far from being classified, and new ways to construct partially hyperbolic examples have been devised \cite{GHO}, which depend to some extent on Anosov systems, but seem likely to be more flexible and maybe combinable with the techniques in \cite{BGHP}. Even the most basic questions in high dimensions remain quite open. 

We refer the reader to  \cite{BDV} for a general overview of smooth dynamics and to \cite{Wilkinson} for a recent account on partial hyperbolicity. In \cite{CHHU} the reader can find a survey on the dynamics of partially hyperbolic diffeomorphisms specialized to dimension 3 which also touches upon the classification problem. 

If the reader wishes to know more about the classification problem of partially hyperbolic diffeomorphisms in dimension 3, then the following references could be a useful introduction \cite{CHHU, HP, PotrieICM, BFFP, BFP}. 
%
%%%%%%%%%%%%%%%%%%%%%%%%
%\section{Further reading}\label{s.furtherreading}
%
%For more information on the theory of foliations we recommend \cite[Chapter 4]{Calegari} for an intuitive and geometric introduction. The topological structure of Anosov flows have received a lot of attention and development since the pioneering work of Margulis and Plante-Thurston, we refer the reader to \cite{Barthelme} for a survey of the main results in dimension 3 together with several of the key ideas.   
%
%Partially hyperbolic diffeomorphisms appeared not only as a generalisation of time one maps of Anosov flows and have played a prominent role in smooth dynamics in the recent years. We refer the reader to  \cite{BDV} for a general overview of smooth dynamics and to \cite{Wilkinson} for a recent account on partial hyperbolicity. In \cite{CHHU} the reader can find a survey on the dynamics of partially hyperbolic diffeomorphisms specialized to dimension 3 which also touches upon the classification problem. 
%
%If the reader wishes to know more about the classification problem of partially hyperbolic diffeomorphisms in dimension 3, then the following references could be a useful introduction \cite{CHHU, HP, PotrieICM, BFFP, BFP}. 
%

%%%%%%%%%%%%%%%%%%%%%%%%%%%%%%%%%%%%%%%%%%%%%%%%%%%

%TAMANHO DE LETRA

%\tiny  5   5
%\scriptsize    7   7
%\footnotesize  8   8
%\small 9   9
%\normalsize    10  10
%\large 12  12
%\Large 14  14.40
%\LARGE 18  17.28
%\huge  20  20.74
%\Huge  24  24.88

\end{document}